\documentclass[11pt]{amsart}

\usepackage[english]{babel}
\usepackage[utf8]{inputenc}
\usepackage[T1]{fontenc}

\usepackage{amsmath,amssymb,amscd}
\usepackage{ifthen,srcltx}
\usepackage{xcolor}

\usepackage{hyperref}

\title{Betti numbers of Shimura curves and arithmetic three--orbifolds}

\author{Miko\l{}aj Fr\k{a}czyk and Jean Raimbault}
  \address{Alfr\'ed R\'enyi Institute of Mathematics,  Re\'altanoda utca 13-15, H-1053, Budapest, Hungary}
  \email{fraczyk@renyi.hu}
  \address{Institut de Math\'ematiques de Toulouse ; UMR5219 \\ Universit\'e de Toulouse ; CNRS \\ UPS IMT, F-31062 Toulouse Cedex 9, France}
  \email{Jean.Raimbault@math.univ-toulouse.fr}

\thanks{
The writing of this paper was facilitated by a visit of the second author to the R\'enyi Institute, funded by ERC Consolidator Grant 648017 to Mikl\'os Ab\'ert. The second author also benefited from support from the grant ANR-16-CE40-0022-01 - AGIRA. First author was supported by ERC Consolidator Grant 648017.}

\newcommand{\sub}{\mathrm{Sub}}
\newcommand{\eps}{\varepsilon}

\newcommand{\ovl}[1]{\overline #1}

\newcommand{\pl}{\partial}
\newcommand{\bs}{\backslash}

\newcommand{\vol}{\operatorname{vol}}

\newcommand{\tr}{\operatorname{tr}}
\newcommand{\otr}{\operatorname{Tr}}

\newcommand{\PSL}{\mathrm{PSL}}
\newcommand{\PU}{\mathrm{PU}}
\newcommand{\PGL}{\mathrm{PGL}}

\newcommand{\SU}{\mathrm{SU}}
\newcommand{\SO}{\mathrm{SO}}

\newcommand{\N}{\mathbf{N}}

\newcommand{\abs}{\mathrm{abs}}

\newcommand{\isom}{\mathrm{Isom}}

\newcommand{\NN}{\mathbb N}

\newcommand{\CC}{\mathbb C}
\newcommand{\RR}{\mathbb R}
\newcommand{\ZZ}{\mathbb Z}
\newcommand{\HH}{\mathbb H}

\newcommand{\QQ}{\mathbb Q}

\newcommand{\showcomments}{yes}

\newsavebox{\commentbox}
\newenvironment{comment}%
{\ifthenelse{\equal{\showcomments}{yes}}%
{\footnotemark
        \begin{lrbox}{\commentbox}
        \begin{minipage}[t]{1.25in}\raggedright\sffamily\tiny
        \footnotemark[\arabic{footnote}]}
{\begin{lrbox}{\commentbox}}}
{\ifthenelse{\equal{\showcomments}{yes}}
{\end{minipage}\end{lrbox}\marginpar{\usebox{\commentbox}}}
{\end{lrbox}}}

\numberwithin{equation}{section}

\newtheorem{theo}{Theorem}[section]
\newtheorem{lem}[theo]{Lemma}
\newtheorem{prop}[theo]{Proposition}
\theoremstyle{definition}

\newtheorem{theostar}{Theorem}

\begin{document}

\maketitle

\begin{abstract}
  We show that asymptotically the first Betti number $b_1$ of a Shimura curve satisfies the Gauss--Bonnet equality $2\pi(b_1 - 2) = \vol$ where $\vol$ is hyperbolic volume; equivalently $2g - 2 = (1+o(1))\vol$ where $g$ is the arithmetic genus. We also show that the first Betti number of a congruence hyperbolic 3--orbifolds asymptotically vanishes relatively to hyperbolic volume, that is $b_1/\vol \to 0$. This generalises results from \cite{7samurai} and \cite{fraczyk} and we rely on results and techniques from these works, most importantly the notion of Benjamini--Schramm convergence of locally symmetric spaces. 
\end{abstract}

\section{Introduction}

\subsection{Benjamini--Schramm convergence}

Let $G$ be a semisimple Lie group, $K \subset G$ a maximal compact subgroup and $X = G/K$ the associated symmetric space. Benjamini--Schramm convergence of locally symmetric orbifolds $\Gamma \bs X$ of finite volume was introduced in \cite{7samurai}. The Benjamini--Schramm convergence of a sequence of finite volume locally symmetric spaces $(\Gamma_i\bs X)_{i\in\mathbb N}$ to the symmetric space $X$ is equivalent to the following simple geometric condition:   
\begin{equation} \label{BS_symmetric}
  \forall R>0,\, \lim_{i\to\infty}\frac{\vol((\Gamma_i \bs X)_{<R})}{\vol(\Gamma_i \bs X)}=0, 
\end{equation}
where $M_{<R}$ denotes the $R$-thin part of a Riemannian orbifold $M$ (which we take to include the full singular set, see \eqref{def_thinpart} below). 

In addition to $X$ there are other possible limits in the Benjamini-Schramm topology. In order to describe them it is convenient to pass to the language of invariant random subgroups (IRS) of the group $G$. These are the Borel probability measures on the Chabauty space $\sub_G$ of closed subgroups which are invariant under conjugation by elements of $G$. For every lattice $\Gamma$ of $G$  there is a unique $G$-invariant probability measure on $G/\Gamma$ and its pushforward by the map $g\Gamma \mapsto g\Gamma g^{-1}$ gives an IRS denoted $\mu_\Gamma$. It was observed in \cite{7samurai} that $(\Gamma_i\bs X)$ converges to $X$ if and only if $\mu_{\Gamma_i}$ converge weakly-* to the trivial IRS $\delta_{\{1\}}.$ In general a sequence $(\Gamma_i\bs X)$ converges Benjamini-Schramm if and only if $\mu_\Gamma$ converges weakly-* to some IRS $\nu$. The limit IRS $\nu$ is always supported on discrete subgroups and the Benjamini-Schramm limit is the random locally symmetric  space $X/\Lambda$ where $\Lambda$ is a $\nu$-random subgroup of $G$.

It was proven in \cite{7samurai}, as a consequence of the Nevo--Stück--Zimmer theorem, that if $G$ is semisimple of higher rank, with all factors having property (T) then any sequence of irreducible locally symmetric spaces converges in the Benjamini--Schramm sense to $X$. This was extended to all nontrivial products in \cite{levit} (see also \cite{Matz} for more precise results in a very specific case). 

This statement is known to be false when $G = \SO(n, 1)$ or $\SU(n, 1)$, because in those cases there are lattices $\Gamma\subset G$ such that $H^1(\Gamma, \mathbb R)\neq 0$ (see \cite{Millson}, \cite{Li_Millson}, \cite{Kazhdan1977}). On the other hand restricting attention to the family of {\em arithmetic congruence lattices} in $G$ (see \ref{lattices} below for a short description) the first author proved in \cite{fraczyk} that for $G=\SO(2,1),\SO(3,1)$ the symmetric space $X = \HH^2, \HH^3$ is the only possible limit in the Benjamini-Schramm topology for a sequence of congruence lattices. Previously the second author \cite{raimbault} had proven a similar result for the family of non-uniform, not necessarily torsion-free lattices (nonuniformity makes them much easier to deal with algebraically). In this paper we remove the torsion-free hypothesis in general. 

\begin{theostar} \label{Main}
  If $G = \PGL_2(\RR)$ or $\PGL_2(\CC)$ and $\Gamma_n$ is a sequence of irreducible arithmetic lattices in $G$, which are either all congruence and pairwise distinct, or pairwise non-commensurable, then the sequence of locally symmetric spaces $\Gamma_n \bs X$ converges in the Benjamini--Schramm sense to $X$. 
\end{theostar}

In \cite{fraczyk} the torsion free assumption was necessary because the methods only allowed to control the volume of the subset of thin part consiting of the collars of short geodesics. For a sequence of general arithmetic congruence orbifolds $(\Gamma_n\bs X)_{n\in\N}$ it could \textit{a priori} happen that the vast majority of the thin part comes from the cusps or the conical singularities so the sequence does not converge to $X$. Theorem \ref{Main} excludes this possibility. For the proof we use the estimates developped in \cite{fraczyk} to show that any weak-* limit of the sequence $\mu_{\Gamma_n}$ is supported on elementary subgroups. By \cite{Osin} the only IRS supported on this set is the trivial IRS, hence the theorem. We carry out the second step of this scheme of proof in detail in Proposition \ref{main_tech}, which is valid for all sequences of lattices in proper Gromov-hyperbolic spaces. 

We note that because we are using a soft method our approach does not indicate the rate of decay of $\vol((\Gamma_n\bs X)_{<R})/\vol(\Gamma_n\bs X)$ as opposed to \cite{fraczyk}.


\subsection{Genus of Shimura curves}

One application of Theorem \ref{Main} is to determine the asymptotic genus of congruence surfaces of large volume. For compact surfaces without singularities the genus and volume are essentially linearly related by the Gauss-Bonnet formula. However for 2-orbifolds terms coming from cone points and cusps appear in the formula, and it is easy to see that there exists sequences of hyperbolic orbifolds with underlying space a sphere and volume going to infinity. This also has an algebraic interpretation: if $S$ is isomorphic as a Riemann surface to the $\CC$-points of an algebraic variety defined over a number field, which is the case for orbifolds obtained from congruence groups (so-called Shimura curves, see \cite{Shimura}), then its geometric genus is given by the Riemann--Hurwitz formula and essentially proportional to the volume while its arithmetic genus equals the topological genus of the underlying surface and can be arbitrarily smaller than the former. 

It is known that this phenomenon cannot occur for congruence orbifolds: using the uniform spectral gap for congruence quotients (see \cite{Clozel_tau}) and a theorem of P.~Zograf \cite{Zograf} it follows that there is a lower bound of the form $g \ge c \vol$ for congruence subgroups (see also \cite{LMR_genus_zero}). As a consequence of Theorem \ref{Main} we obtain the following asymptotically more precise result (we note that it was known for congruence covers of the modular surface by a result of J.~G.~Thompson \cite{Thompson}). 

\begin{theostar} \label{genus}
  Let $\Gamma_n$ be a sequence of congruence lattices in $\PSL_2(\mathbb R)$, and let $g_n$ be the topological genus of the orbifold $O_n = \Gamma_n \bs \HH^2$. Then, assuming the $\Gamma_n$ are not pairwise conjugated, we have
  \[
  \lim_{n\to +\infty} \frac{g_n}{\vol O_n} = \frac 1{4\pi}.
  \]
\end{theostar}

\medskip


\subsection{Betti numbers of 3--orbifolds}

Theorem \ref{genus} implies the weaker result that $b_1(\Gamma_n)/\vol(\Gamma_n \bs \HH^2)$ converges to $1/2\pi$ for a sequence of congruence lattices. Indeed, the rank of abelianisation is essentially equal to twice the genus in a BS-convergent sequence. This can be proven more directly by analytical means, as $1/2\pi$ is the first $L^2$-Betti number of the hyperbolic plane. While more complicated, the analytic approach generalizes to the dimension 3 and where obtain the following result.

\begin{theostar} \label{Betti}
  Let $\Gamma_n$ be a sequence of congruence lattices in $\PSL_2(\mathbb C)$. Then
  \[
  \lim_{n \to +\infty} \frac{b_1(\Gamma_n)}{\vol(\Gamma_n \bs \HH^3)} = 0.
  \]
\end{theostar}

This was proven in \cite{raimbault} for non-uniform lattices, and in \cite{fraczyk} in the case of all torsion-free lattices. Our proof is very similar to the proof for hyperbolic 3--manifolds appearing in \cite{7samurai}. 


\subsection{Congruence lattices}\label{lattices}

For completeness we give an explicit description of the congruence arithmetic latices in $G={\rm PGL}(2,\RR),{\rm PGL}(2,\CC)$, though we will not directly use this structure theory in the rest of the paper. Let $\mathbb K=\RR,\CC$. We start by choosing a number field $k$ with Archimedean places $\nu_1,\ldots,\nu_d$ such that $k_{\nu_1}\simeq \mathbb K$ and $k_{\nu_i}\simeq \mathbb R$ for $i\geq 2$. In what follows $\mathbb A,\mathbb A_f$ stand for the ring of ad\`eles, respectively finite ad\`eles of $k$. We will write $k\ni x\mapsto (x)_\nu\in k_\nu$ for the embedding of $k$ in its completion $k_\nu$.  Let $a,b\in k^\times$ be such that $(a)_{\nu_i},(b)_{\nu_i}$ are positive for $i\geq 2$ and $(a)_{\nu_1}$ or $(b)_{\nu_1}$ is negative if $\mathbb K\simeq \RR$. We define the quaternion algebra $A$ as 
\[
A = k+ {\mathbf i}k+{\mathbf j}k+ \mathbf{ij}k,
\]
subject to the relations ${\mathbf i}^2=-a,{\mathbf j}^2=-b,\mathbf{ij}=-\mathbf{ji}$. By our choice of $a,b$ we have $A\otimes_k k_{\nu_1}\simeq M(2,\mathbb K)$ and for $i\geq 2$ the algebra $A\otimes_k k_{\nu_i}$ is isomorphic to the Hamilton's quaternions. We form an algebraic group ${\mathrm{PA}^\times}=A^\times/k^\times$. It is an adjoint simple group of type $A_1$ defined over $k$. Note that $\mathrm{PA}^\times(\mathbb A)=\mathrm{PA}^\times(k\otimes_\QQ \RR)\times \mathrm{PA}^\times(\mathbb A_f)$ and
\[
\mathrm{PA}^\times(k\otimes_\mathbb{Q} \mathbb R)=\prod_{i=1}^d\mathrm{PA}^\times(k_{\nu_i})\simeq {\rm PGL}(2,\mathbb K)\times {\rm PO}(3)^{d-1}.
\]
Choose an open compact subgroup $U$ of $\mathrm{PA}^\times(\mathbb A_f).$ Let $\Gamma_U=\mathrm{PA}^\times(k)\cap (\mathrm{PA}^\times(k\otimes_\QQ \RR)\times \mathrm{PA}^\times(\mathbb A_f))$. By a classical result of Borel-Harish-Chandra \cite{BHC} the group $\Gamma_U$ is a lattice in $\mathrm{PA}^\times(k\otimes_\QQ \RR)\times \mathrm{PA}^\times(\mathbb A_f)\simeq {\rm PGL}(2,\mathbb K)\times {\rm PO}(3)^{d-1}\times U$.  The projection of $\Gamma_U$ to the factor $\PGL(2,\mathbb K)$ is a {\em congruence arithmetic lattice} in ${\rm PGL}(2,\mathbb K)$. Every congruence arithmetic lattice of ${\rm PGL}(2,\mathbb K)$ arises in this way.


\subsection{Outline of the paper} 

In Section \ref{sec_BSconv} we describe a general criterion for the Benjamini--Schramm convergence of lattices in the isometry group of a proper Gromov-hyperbolic spaces and we apply it, together with the estimates from \cite{fraczyk}, to deduce Theorem \ref{Main}. Next, in section \ref{sec_sing} we give a precise metric description of the singular locus of hyperbolic 2- and 3-orbifolds, and a way to smooth the boundary of the thick part while keeping control of the geometry (the technical details of which are left to Appendix \ref{appendix_smooth}). We use the description of singularities and Theorem \ref{Main} to deduce Theorem \ref{genus} in section \ref{sec_genus}. In section \ref{sec_betti} we use heat kernel methods (for which we need the precise description of the smoothed thick part) to deduce Theorem \ref{Betti} from Theorem \ref{Main}. 


\section{Benjamini--Schramm convergence of quotients of hyperbolic spaces}  \label{sec_BSconv}

\subsection{Orbital integrals on hyperbolic spaces}

Let $X$ be a proper Gromov-hyperbolic space and $G = \isom(X)$. With the compact-open topology $G$ is a locally compact second countable topological group. For $\gamma \in G$ we denote by $G_\gamma$ its centraliser. The following lemma is a slight generalisation of \cite[Corollary 3.10(2) on p. 463]{Bridson_Haefliger}---the latter dealing only with discrete groups. It might be possible to straightforwrdly adapt the arguments in loc. cit. to our case, but we give a different, mostly self-contained proof. 

\begin{lem} \label{cocompact_centr}
  Let $\gamma \in G$ be an hyperbolic isometry. Then $G_\gamma / \langle \gamma \rangle$ is compact. 
\end{lem}

For the proof we use the following lemma, which should be standard but we could not find in the literature. The proof is a bit long and technical and we put it in Appendix \ref{proof_lemma}. 

\begin{lem} \label{dist_trans}
  Let $\gamma$ be an hyperbolic isometry of $X$. For any $x \in X$ there exists constants $C = C(x, \gamma, \delta)$ and $A = A(x, \gamma, \delta)$ such that for any $y \in X$ and any $k$ sufficiently large (depending on $\gamma, x, \delta$) we have
  \[
  d(y, \gamma^k y) \ge Ck + 2d(y, \langle \gamma\rangle x) - A.
  \]
\end{lem}

\begin{proof}[Proof of lemma \ref{cocompact_centr}]
  Let $\tau = d(\gamma) := \inf\{d(y,\gamma y)|y\in X\}$ be the minimal displacement of $\gamma$. Fix $x \in X$, let $k, A, C$ as given by Lemma \ref{dist_trans} and define: 
  \[
  D = \{y\in X| d(y,\gamma^k y)\leq k\tau + 1\}.
  \]
  It is a non-empty (by definition of $\tau$) closed $G_\gamma$-invariant subset of $X$. Given that the action of $G_\gamma$ on $D$ is proper, the Lemma will follow once we prove that $\langle \gamma\rangle \bs D$ is compact. The previous Lemma implies that
  \[
  D \subset \{ y \in X :\: d(y, \langle \gamma\rangle x) \le (\tau - C)k + A + 1 \}
  \]
  so that $D \subset \gamma^\ZZ B(x, R)$ for some sufficiently large $R$, and as $X$ is proper this in turn implies that $\langle \gamma\rangle \bs D$ is compact. 
\end{proof}

Let $dg$ be a fixed Haar measure on $G$. According to the lemma above the subgroup $G_\gamma$ admits a lattice so it is unimodular and we have a decomposition $dg = dx dh$ where $dx$ is a $G$-invariant measure on $G / G_\gamma$ and $dh$ a Haar measure on $G_\gamma$, both depending only on the original choice of $dg$. For a function $f \in C_0(G)$ we can then define the {\em orbital integral} associated to $\gamma$ by:
\begin{equation}\label{orb_int}
  \mathcal O_f(\gamma) = \int_{G/G_\gamma} f(\gamma^{-1} x \gamma) dx
\end{equation}
which depends only on the $G$-conjugacy class $[\gamma]_G$. 


\subsection{General criterion for Benjamini--Schramm convergence}

Here again $X$ is always a proper Gromov-hyperbolic space and $G = \isom(X)$. We assume that the action of $G$ on $X$ is non-elementary. The {\em elliptic radical} of $G$ can then be defined as its unique maximal normal compact subgroup (see \cite[Proposition 3.4]{Osin}; in our context, by properness of $X$ means that bounded elements are the same as compact ones). The following lemma is a special case of \cite[Theorem 1.5]{Osin}. 

\begin{lem} \label{full_limit_set}
  Let $\mu$ be an invariant random subgroup of $G$. Then either $\mu$ is supported on the elliptic radical or it has full limit set. 
\end{lem}

Recall from \cite[Section 3]{Gelander_ICM} that there is a ``Benjamini--Schramm topology'' on the set of Borel probability measures on the Gromov--Hausdorff space of pointed proper metric spaces (up to isometry). The set of measures supported on spaces locally isometric to $X$ is precompact in this topology. Moreover, if $X$ is a locally symmetric space then \eqref{BS_symmetric} is equivalent to $\Gamma_i \bs X$ converging in the Benjamini--Schramm topology to $X$. 

There is a continuous injective map from the space of invariant random subgroups of $G$ to the Benjamini--Schramm space. If $\Gamma_i$ are lattices in $G$ then the sequence of uniformly pointed spaces $\Gamma_i \bs X$ converges to $X$ if and only if the IRSs $\mu_{\Gamma_i}$ converge to the trivial IRS. We will use this to prove the following criterion for convergence. 

\begin{prop} \label{main_tech}
  Let $U$ the set of hyperbolic isometries in $G$. Assume that the elliptic radical of $G$ is trivial. If $\Gamma_n$ is a sequence of lattices in $G$ which satisfies:
  \begin{equation} \label{regular_sums}
    \lim_{n \to +\infty} \frac{\sum_{[\gamma]_{\Gamma_n} \subset U} \vol((\Gamma_n)_\gamma\bs G_\gamma)\mathcal O_f(\gamma)}{\vol(\Gamma_n \bs G)} = 0  
  \end{equation}
  then the sequence of metric spaces $\Gamma_n \bs X$ converges to $X$ in the Benjamini--Schramm topology.
\end{prop}
 
\begin{proof}
  Let $\mu_n$ be the invariant random subgroup of $G$ supported on the conjugacy class of $\Gamma_n$. We want to prove that any weak limit $\mu$ of a subsequence of $(\mu_n)$ is equal to the trivial IRS $\delta_e$. By Lemma \ref{full_limit_set}, and the fact that a subgroup of $G$ containing no hyperbolic isometries has at most one limit point (cf. \cite[Section 8.2]{Gromov_hyp}) it suffices to prove that any such $\mu$ contains no hyperbolic isometries. 

  To prove this choose a covering $U = \bigcup_{C \in \mathcal C} C$ of $U$ where $\mathcal C$ is countable and every $C \in \mathcal C$ is compact. We can do this since $\sub_G$ is metrizable \cite[Proposition 2]{dlHarpe1}. Let $W_C = \Lambda : \Lambda \cap C \not= \emptyset$ ; this is a Chabauty-closed subset of $\sub_G$. If $\nu$ is a nontrivial IRS then by Lemma \ref{full_limit_set} and previous paragraph it almost surely contains a hyperbolic element. Hence, there is $C \in \mathcal C$ such that $\nu(W_C) > 0$. We need to prove the opposite for $\mu$, which amounts to the following : for every $C$ there exists a non-negative Borel function $F$ on $\sub_G$ which is positive on $W_C$ and such that $\int_{\sub_G} F(\Lambda) d\mu(\Lambda) = 0$. 

  Let us fix $C \in \mathcal C$ and prove this. There exists an open relatively compact subset $V$ with $C \subset V$ and  $\ovl V\subset U$. Choose any $f \in C^\infty(G)$ such that $f > 0$ on $C$ and $f = 0$ on $G \setminus V$ and define :
  \[
  F(\Lambda) = \begin{cases} \sum_{\lambda\in\Lambda} f(\lambda) &\text{if } \Lambda \text{ is discrete; } \\
    1 & \text{ if } \Lambda \text{ is not discrete and intersects } C;\\ 
    0 & \text{otherwise.} \end{cases}
  \] 
  Then $F$ is lower semicontinuous on $\sub_G$, non-negative and positive on $W_C$. On the other hand we have :
  \begin{align*}
    \int_{\sub_G} F(\Lambda) d\mu_n(\Lambda) &= \frac 1 {\vol(\Gamma_n \bs G)} \int_{G/ \Gamma_n} \sum_{\gamma \in g\Gamma_n g^{-1}} f(\gamma) dg \\
        &= \frac 1 {\vol(\Gamma_n \bs G)} \sum_{[\gamma]_{\Gamma_n} \subset U} \vol((\Gamma_n)_\gamma\bs G_\gamma)\mathcal O_\gamma(f). 
  \end{align*}
  By the so-called ``Portemanteau theorem'' \cite[Theorem 13.16]{klenke} the limit inferior of the left-hand side is larger or equal to $\int_{\sub_G} F(\Lambda) d\mu(\Lambda)$. By \eqref{regular_sums} we have that the right-hand side converges to 0. It follows that
  \[
  \int_{\sub_G} F(\Lambda) d\mu(\Lambda) = 0
  \]
  which finishes the proof. 
\end{proof}


\subsection{Proof of Theorem \ref{Main}}

If $X$ is a rank-one irreducible symmetric space such as $\HH^2$ or $\HH^3$ and $G = \isom(X)$ then $G$ is a simple Lie group of non-compact type and its elliptic radical is trivial. Theorem \ref{Main} thus follows immediately from Proposition \ref{main_tech} and the following result extracted from \cite{fraczyk}. 

\begin{theo}\label{reg_conv}
  Let $G = \PGL_2(\RR)$ or $\PGL_2(\CC)$ and let $U$ be the set of hyperbolic elements of $G$. Let $\Gamma_n$ a sequence of arithmetic congruence lattices in $G$, such that $\vol(\Gamma_n \bs G) \to +\infty$ or any sequence of pairwise non-commensurable arithmetic lattices. Then for any $f \in C_0^\infty(G)$ we have : 
  \begin{equation} \label{sum_hyp_integrals}
  \frac 1 {\vol(\Gamma_n \bs G)} \left| \sum_{[\gamma] \subset U} \vol((\Gamma_n)_\gamma\bs G_\gamma)\mathcal O_f(\gamma) \right| \xrightarrow{n \to +\infty}{} 0. 
  \end{equation}
\end{theo}

\begin{proof}
  If $\Gamma$ is an arithmetic lattice in $\PGL_2(\RR)$ or $\PGL_2(\CC)$ then an element $\gamma \in \Gamma$ is hyperbolic if and only if it is semisimple and of infinite order.  In the proof of \cite[Theorem 1.8]{fraczyk}, starting form the lines (10.7-10.9) the author bounds the sum 
\begin{equation}\label{eqUpperBound}\sum_{\substack{[\gamma]_{\Gamma}\\ \textrm{non torsion}}}\vol(\Gamma_\gamma\bs G_\gamma)\mathcal O_\gamma(f)
\end{equation} for congruence arithmetic lattices. The line (10.7) of \cite[p. 67]{fraczyk} is the ad\`elic version of the last sum where we group together the classes conjugate over $\mathrm{PA}^\times(k)$, where $\mathrm{PA}^\times$ is the group used to construct the lattice $\Gamma$ as explained in Section \ref{lattices}. The passage between the ad\`elic and classical trace formula is explained in \cite[Theorem 4.21]{fraczyk}. Proceeding as in \cite[p. 67-69]{fraczyk} we obtain the bound 
$$\sum_{\substack{[\gamma]_{\Gamma}\\ \textrm{non torsion}}}\vol(\Gamma_\gamma\bs G_\gamma)\mathcal O_\gamma(f)\ll \vol(\Gamma\bs G)^{0.986}.$$
Any hyperbolic conjugacy class $[\gamma]_\Gamma$ is non-torsion so we can deduce the that the sum \eqref{sum_hyp_integrals} converges to 0 as $\vol(\Gamma\bs X)\to\infty$ and $\Gamma$ is a congruence arithmetic lattice. In order to establish the convergence for sequences of pairwise non-commensurable arithmetic lattices $(\Gamma_n)_{n\in\NN}$ we choose for each $n$ a maximal arithmetic lattice $\Lambda_n$ containing $\Gamma_n$.  It is always a congruence arithmetic lattice. We have 
\begin{align*}\frac{1}{\vol(\Gamma_n\bs X)}\left| \sum_{[\gamma]_{\Gamma_n}\in U}\vol((\Gamma_n)_\gamma\bs G_\gamma)\mathcal O_f(\gamma) \right|&\leq \\
\frac{1}{\vol(\Gamma_n\bs X)}\sum_{[\gamma]_{\Gamma_n}\in U}\vol((\Gamma_n)_\gamma\bs G_\gamma)\mathcal O_{|f|}(\gamma) &\leq \\
  \frac{1}{\vol(\Lambda_n\bs X)}\sum_{[\gamma]_{\Lambda_n}\in U}\vol((\Lambda)_\gamma\bs G_\gamma)\mathcal O_{|f|}(\gamma)&=o(1).
\end{align*}
\end{proof}


\section{Structure of the singular locus of closed hyperbolic orbifolds} \label{sec_sing}

To be able to deduce from the sole Benjamini--Schramm convergence of a sequence of orbifolds further asymptotic results on topological invariants we need a fine metric description of the singular locus. The results in this section provide it; they are not really original but precise statements such as we need are not found in the litterature. As usual our main tool is the Margulis lemma. 

\begin{theo} \label{margulis}
  For every $n \ge 2$ there exists $\eps = \eps(n) > 0$ such that the following holds. Let $\Gamma$ be a discrete subgroup of isometries of $\HH^n$, then for any $x \in \HH^n$ the subgroup
  \[
  \Gamma_\eps := \langle \gamma \in \Gamma : d(x, \gamma x) \le \eps \rangle
  \]
  is virtually abelian.
\end{theo}

In the sequel we will only work in 2 or 3-dimensional hyperbolic space, and we let $\eps$ denote a Margulis constant which is valid for both cases. Recall that $O_{\le \eps}$ stands for the $\eps$-thin part of an orbifold $O$, for which we use the following definition: if $O = \Gamma \bs X$ where $X$ is the orbifold universal cover and we assume $X$ to be CAT(0) then  
\begin{equation} \label{def_thinpart}
  O_{\le \eps} = \Gamma \bs \{ \tilde x \in X :\: \exists \gamma \in \Gamma \setminus\{\mathrm{Id}\},\, d(\tilde x, \gamma\tilde x) \le \eps \}
\end{equation}
which includes the singular locus of $O$---note that in the litterature, e.g. in \cite{BMP}, a different convention is often used where only points with large stabilisers are included. The closure of the complement of $O_{\le \eps}$ (the $\eps$-thick part) will be denoted by $O_{\ge \eps}$.

In fact we need to tweak a bit the definition of the thin part around that part of the singular locus where the cone angle is $\pi$: around these vertices or geodesics we put a collar whose width is $\eps/6$ (instead of $\eps/2)$.


\subsection{2-dimensional orbifolds}

In $\PGL_2(\RR)^+$ all the virtually abelian discrete subgroups are given by the following list: 
\begin{enumerate}
\item An infinite cyclic group generated by an hyperbolic or parabolic isometry;

\item \label{isolated_cone_pt} A finite cyclic group generated by an elliptic isometry;

\item \label{pair_half_turns} An infinite dihedral group generated by two elliptic isometries of order 2. 
\end{enumerate}
As a first consequence we see that the singular locus of an orientable hyperbolic 2-orbifold consists only of {\em cone points}, that is all non-manifold points have a neighbourhood which is isometric to the quotient of a disc by a finite cyclic group. 

In addition we can deduce from this classification a metric description of the singular locus. We need the following notation: given an elliptic isometry $\gamma$ with fixed point $x$ and rotation angle $\theta$, let $\ell(\theta, \eps)$ be the smallest $\ell$ such that $d(y, \gamma y) \ge \eps$ for $d(x, y) = \ell$. Similarly, given a hyperbolic isometry $\gamma$ of minimal displacement $\ell$ we define $r(\ell, \eps)$ to be the minimal distance from its axis at which an hyperbolic isometry translates of at least $\eps$. 

\begin{lem} \label{2dim_sing}
  Let $O = \Gamma \bs \HH^2$ be an orientable hyperbolic 2-orbifold and $x$ a point in its singular locus. Then $x$ is an isolated cone point and one of the following possibilities hold:
\begin{enumerate}
 \item If its angle is $2\pi/m$ with $m \ge 3$ then there is no other singular point in the ball $B_O(x, \ell)$ where $\ell = \ell(2\pi/m, \eps)$. 

\item If the angle is equal to $\pi$ then either there is no other singular point within distance $\ell(\pi, \eps)$, or there is one (and its cone angle is also $\pi$) at distance $\ell_x < \ell(\pi, \eps)$ but no other within distance $r(\ell_x, \eps)$ of $x$.
\end{enumerate}
\end{lem}

\begin{proof}
  Let $\Gamma x\in O$ be as in the statement, with $x\in\HH^2$. Then $x$ is a fixed point of a non-trivial element of $\Gamma$, and it follows that the subgroup
  \[
  \Gamma_x^\varepsilon = \{ \gamma \in \Gamma :\: d(x, \gamma x) \le \eps \}
  \]
  must be one of those described in (\ref{isolated_cone_pt}) or (\ref{pair_half_turns}) at the beginning of this section; let $\gamma_0$ be a generator (with minimal rotation angle) of the cyclic subgroup fixing $x$ and $m > 1$ its order.

  In any case $x$ lies above a conical point in $O$. Assume now that $m \ge 3$; then $\Gamma_x = \langle \gamma_0 \rangle$ and by the Margulis lemma there is no other fixed point of a non-trivial element in $\Gamma$ within the set
  \[
  C = \{ y \in \HH^2 :\: d(y, \gamma_0 y) \le \eps) \}.
  \]
  By definition the ball $B_{\HH^2}(x, \ell(2\pi/m, \eps))$ is contained in $C$, so it contains no other singular point. 

  If $m = 2$ and there is another elliptic fixed point $x' \in \HH^2$ with $d(x, x') \le \ell(\pi, \eps)$ then we might assume that $x'$ is the closest such point. By the previous paragraph any nontrivial $\gamma_0' \in \Gamma$ fixing $x'$ must be of order $2$. Let $\eta = \gamma_0\gamma_0'$. It is a hyperbolic isometry with axis containing the geodesic $\alpha$ joining $x$ to $x'$ and translation distance $2d(x, x')$. Write $\Gamma_\alpha$ for the setwise stabilizer of $\alpha$ in $\Gamma$. For every $\gamma\in\Gamma_\alpha$ not fixing $x$ we will have $d(x,\gamma x)\geq 2d(x,x')$ as otherwise $\gamma_0\gamma$ would have a fixed point closer to $x$ than $x'$.  We deduce that $\Gamma_\alpha=\langle\gamma_0, \gamma_0'\rangle$.  The former is a maximal virtually abelian subgroup of $\Gamma$ (it is an intersection of $\Gamma$ with the normaliser of a split torus). The Margulis lemma now implies that within the ball $B_{\HH^2}(x, \ell(\pi, \eps))$ (resp. $B_{\HH^2}(x, r(\ell_x, \eps))$) any other elliptic fixed point must be a translate of either $x$ or $x'$ by a power of $\eta$, as any such point is moved by at most $\eps$ by $\gamma_0$ (resp. $\eta$) and hence its stabiliser in $\Gamma$ must belong to $\Gamma_\alpha$. 
\end{proof}


\subsection{3-dimensional orbifolds}

\subsubsection{Description of the singular locus}

The list of discrete virtually abelian subgroups of $\PGL_2(\CC)$ is long enough to make us avoid giving a complete description. Rather, we will assume that $\Gamma$ is a cocompact lattice in $\PGL_2(\CC)$ and $\Lambda$ a maximal virtually abelian subgroup of $\Gamma$ which contains torsion elements (which is all we need to prove Theorem \ref{Betti}). If $\Lambda$ contains a hyperbolic element $\gamma$ then it must normalise $\langle\gamma\rangle$, so it is contained in the normalizer of a maximal torus. Any such normalizer is isomorphic to $\CC^\times \rtimes \ZZ/2$. Otherwise $\Lambda$  contains only elements if finite order and so by Burnside's theorem it must be a finite subgroup of the maximal compact $\PU(2)$. It follows that $\Lambda$ is one of the following groups:
\begin{enumerate}
\item $\langle \gamma, \eta \rangle \cong \ZZ \times \ZZ/m$ where $\gamma, \eta$ are respectively hyperbolic and elliptic isometries sharing the same axis;

\item $\langle \gamma, \eta, \rho \rangle \cong (\ZZ \times \ZZ/m) \rtimes \ZZ/2$ where $\eta, \gamma$ are as above (with $\eta$ possibly trivial) and $\rho$ is an elliptic of order 2 with axis orthogonal to that of $\gamma$ or $\eta$;

\item One of the finitely many non-dihedral finite subgroups of $\PU(2)$. 
\end{enumerate}
We see from this description that the singular locus of an hyperbolic 3--orbifold consists of closed geodesics (which we'll call {\em singular geodesics}), which can intersect each other. A singular point not on the intersection of two singular geodesics has a neighbourhood isometric to the quotient of a ball by a rotation; the angle of the latter we will call the cone angle of the singular geodesic. We will call a vertex which is at the intersection of two or more singular geodesics a {\em vertex} of the singular locus. 

Together with the Margulis lemma the list above allows us to give the following metric description of the singular locus (see also \cite[Corollary 6.3]{BMP} for a more geometric description, and loc. cit., Fig. 5 on p. 33 for illustrations). This description is analogous to the situation from Lemma \ref{2dim_sing}. 

\begin{lem} \label{3dim_sing}
  Let $O$ be a compact orientable 3-dimensional hyperbolic orbifold and $\Sigma$ its singular locus. Let $x \in \Sigma$ be a vertex. Then one of the two following possibilities hold. 
  \begin{enumerate}
  \item \label{exception} The $\eps/2$-neighbourhood of $x$ is isometric to one of a finite list of orbifolds, whose singular locus has only one vertex and all singular geodesics go through $x$. 

  \item \label{dihedral} There is at most one other singular vertex $x'$ within distance $\eps/2$ of $x$; $x$ and $x'$ are joined by a singular geodesic $c$ of length $\ell$ and cone angle $2\pi/m$, there are two singular geodesics with cone angle $\pi$ and orthogonal to $c$ each going through one of $x$ or $x'$. There are no further components of the singular locus within distance $\max(\ell(2\pi/m, \eps), \ell(r, \eps))$ of $x$ and $x'$. 
  \end{enumerate}
  Moreover if two non-intersecting singular geodesics of $O$ are within distance $\eps/2$ of each other then both have angle $\pi$. 
\end{lem}

\begin{proof}
  Let $O = \Gamma \bs \HH^3$ a closed hyperbolic 3--orbifold. Let $x$ be a vertex in the singular locus of $O$ and $\Pi$ the subgroup of $\Gamma$ fixing a lift $\tilde x$ of $x$ to $\HH^3$. Then $\Pi$ is either a dihedral group $\ZZ/m \rtimes \ZZ/2$ or one of finitely many finite non-dihedral subgroups of $\PU(2)$, according to the list of virtually abelian subgroups of $\Gamma$ above. We note that under the condition in \eqref{exception}.

  If the vertex is as in \eqref{exception} and $\gamma \in \Gamma$, $\gamma \not\in \Pi$ is an elliptic isometry of order $m$ then as (by the Margulis Lemma) $\Pi$ contains all isometries moving $\tilde x$ by more than $\eps$ any fixed point of $\gamma$ must be at distance at least $\ell(2\pi/m, \eps) \ge \ell(\pi, \eps) = \eps/2$ of $\tilde x$. Similarly any hyperbolic isometry in $\Gamma$ must move $\tilde x$ by at least $\eps$. Hence the quotient $\Pi \bs B(\tilde x, \eps/2)$ embeds into $O$.

  If the vertex has a dihedral stabiliser as in \eqref{dihedral} let $\gamma$ be a generator of the $\ZZ/m$-subgroup and $\eta$ a generator of the $\ZZ$-subgroup commuting with $\gamma$. Then we might assume that either $\ell < \eps/2$ or $m > 5$ (otherwise we can add its neighbourhood to the finite list in \eqref{exception}). Then any elliptic element of $\Gamma$ which does not normalise $\langle\gamma\rangle$ cannot fix a point in $B(\tilde x, \eps)$ (otherwise it and $\gamma$ would generate a subgroup moving a point by less than $\eps$ but not in the list given above, which is not possible by the Margulis Lemma). Similarly it cannot fix a point within $\ell(\ell, \eps)$ of the axis of $\eta$. 
\end{proof}


\subsubsection{Smoothing the thick part}

Let $C = (C_0, C_1, \ldots) \in [0, +\infty[^\NN$. As (a slight variation of) the definition in \cite{Lueck_Schick} we say that a Riemannian manifold has {\em $C$-bounded geometry} if its injectivity radius is at least $C_0$, the normal geodesic flow up to $C_0$ gives coordinates for a collar neighbourhood of the boundary, and the $k$th derivatives of the metric tensor and its inverse (in normal coordinates) are bounded in sup norm by $C_k$. In this section we prove the following lemma. 

\begin{lem} \label{smoothing}
  There exists $C$ such that for any hyperbolic 3--orbifold $O$ there exists a smooth submanifold $O'$ such that: 
  \begin{itemize}
  \item $O_{\ge \eps} \subset O'$ and this is an homotopy equivalence;

  \item $O'$ is of $C$-bounded geometry. 
  \end{itemize}
\end{lem}

We will deduce the lemma from the description of the singular locus and the following general proposition, the proof of which we give in appendix \ref{appendix_smooth}. 

\begin{prop} \label{general_smoothing}
  Let $X$ be a Riemannian $d$-manifold  and $H_1, H_2$ two open subsets whose closures have smooth boundary. Assume the following holds: 
  \begin{itemize}
  \item they intersect transversally in a compact subset; let $\alpha_0$ such that the dihedral angles at the intersection stay within the interval $]\alpha_0, \pi-\alpha_0[$. 

  \item Both manifolds $X \setminus H_i$ are of bounded geometry.
  \end{itemize}
  Then for any $\delta>0$ there exists an open subset $H$ of $X$ such that:
  \begin{enumerate}
  \item \label{contain} $H \supset H_1 \cup H_2$ and they are equal outside of the $\delta$-neighbourhood of $H_1 \cap H_2$;

  \item \label{smooth_bd} the closure of $H$ has a smooth boundary; 

  \item $X \setminus H$ is of bounded geometry; the bounds depend only on $\delta$, on the bounds on the geometry of $X$ and $X \setminus H_i$ and on $\alpha_0$.
  \end{enumerate}
\end{prop}
  
\begin{proof}[Proof of Lemma \ref{smoothing}]
  Observe first that the boundary of the thin part is smooth away from the geodesics with cone angle $\pi$ and the vertices of the singular locus, as follows from the third part of Lemma \ref{3dim_sing}. Thus the non-smooth part of $\pl O_{\ge \eps}$ comes from intersecting tubular neighbourhoods of singular geodesics and short geodesics. There are finitely many possible configurations where the geodesics are not orthogonal to each other (corresponding to case \eqref{exception} of Lemma \ref{3dim_sing}); we do not need to deal in detail with these, so the only problem left to deal with is the following: at all points in the intersection of the tubular neighbourhood $N_1$ (with varying radius) of one geodesic, and the $\eps/6$-tubular neigbourhood $N_2$ of another geodesic orthogonal to the first, the dihedral angle between $\pl N_1$ and $\pl N_2$ stays bounded away from 0 and from $\pi$\footnote{Note that the neighbourhoods corresponding to two geodesics orthogonal to a third one cannot intersect each other, because we took their radius to be $\eps/3$ and the distance between the geodesics outside the $\eps$-thin part is at least $\eps/2$}.

  To prove this note that the maximum and minimum values for these angles both are continuous functions of the radius $0 \le r < +\infty$ of $N_1$. It can be continuously extended to $r = +\infty$, the values then being those of the angle (in a conformal model of $\HH^3$) between $\pl N_1$ and the boundary at infinity of $\HH^3$. As $N_1$ and $N_2$ are never tangent to each other we see by compactness that the maximal and minimal values stay bounded away from $0$ and $\pi$. 
\end{proof}


\section{The genus of congruence orbifolds} \label{sec_genus}

In this section we prove Theorem \ref{genus}. Let $O$ be an hyperbolic orbifold of dimension 2, which is a quotient of the hyperbolic plane $\HH^2$ by a lattice of $\mathrm{PSL}_2(\RR)$. Then the underlying topological space $|O|$ is a surface of finite type, that is it is homeomorphic to a compact surface $S$ with a finite number of points removed. The {\em genus} of $O$ is defined to be the genus of $S$. 

Suppose that $O$ has genus $g$, that it has $k$ punctures and $r$ conical singularities with angles $2\pi/m_1, \ldots, 2\pi/m_r$ (the tuple $(g, k, m_1, \ldots, m_r)$ is then called the {\em signature} of $O$). Then, computing the volume of a well-chosen fundamental polygon we get the following equality (see \cite[Theorem 10.4.2]{Beardon}): 
\begin{equation}\label{GB}
  \vol O = 2\pi\left( 2g - 2 + k + \sum_{i=1}^r \left( 1 - \frac 1{m_i} \right) \right). 
\end{equation}
From this equation we obtain the bound:  
\[
\left| g - \frac{\vol(O)}{4\pi} \right| \le \frac {k + r + 2}{4\pi}. 
\]
We now see that Theorem \ref{genus} follows from Theorem \ref{Main} together with the following proposition.

\begin{prop}
  Let $O_n$ be a sequence of hyperbolic 2--orbifolds which is Benjamini--Schramm convergent to $\HH^2$. Let $k_n, r_n$ be the number of cusps and conical points of $O_n$, respectively. Then $k_n + r_n = o(\vol O_n)$. 
\end{prop}

\begin{proof}
  To prove that $r_n = o(\vol O_n)$ we associate to each conical point $x$ with angle $\theta$ the region
  \[
  \Omega_x = B(x, \ell(\theta, \eps))
  \]
  if there is no other singular point within distance $\ell(\theta, \eps)$. Otherwise let $\ell_x$ be the distance to the nearest singular point and put
  \[
  \Omega_x = B(x, r(\ell_x, \eps)).
  \]
  We will check below the following facts:
  \begin{enumerate}
  \item \label{lowbd_vol} there exists $c > 0$ such that $\vol \Omega_x > c$ for all $n$ and $x \in O_n$;

  \item \label{almost_disjoint} if $x \in O_n$ is a conical point then there is at most one conical point $x' \not= x$ such that $x \in \Omega_{x'}$;

  \item \label{thin_part} for all conical points $x \in O_n$ we have $\Omega_x \subset (O_n)_{\le \eps}$.
  \end{enumerate}
  It follows from these that:
  \[
  r_n \le \frac 1 c\sum_{x \in \Sigma_{O_n}} \vol \Omega_x \le \frac 2 c \vol\left( \bigcup_{x \in \Sigma_{O_n}} \Omega_x \right) \le \frac 2 c \vol(O_n)_{\le\eps}
  \]
  and as the right-hand side is $o(\vol O_n)$ in a BS-convergent sequence we get that $r_n = o(\vol O_n)$.  

  That \ref{thin_part} holds follows immediately from the definitions of $\ell(\theta, \eps)$ and $r(\ell, \eps)$. Point \ref{almost_disjoint}  follows from Lemma \ref{2dim_sing}. 

  It remains to prove \ref{lowbd_vol}. Let $x \in O_n$ be a singularity with cone angle $2\pi/m$ with $m > 2$, let $\tilde x$ be a lift of $x$ to $\HH^2$ and $\ell = \ell(2\pi/m, \eps)$. Then we have
  \[
  \vol(B_{O_n}(x, \ell)) = \frac 1 m B_{\HH^2}(\tilde x, \ell) \gg \frac{e^\ell}m
  \]
  so we need to prove that $e^\ell \gg m$. This follows easily from distance computations in the disk model: by definition of $\ell(\theta, \eps)$ we have that $\ell(\theta, \eps) = \log((1+r)/(1-r))$ where $0 < r < 1$ is such that $d(r, re^{i\theta}) = \eps$. It follows that
  \[
  \cosh(\eps) = 1 + \frac{2r^2|1 - e^{i\theta}|^2}{(1 - r^2)^2}
  \]
  and by standard computations we get that
  \[
  r = 1 - \frac \theta{\sqrt 2\sinh(\eps)} + O(\theta^2)
  \]
  whence it follows that
  \[
  \ell(\theta, \eps) = -\log(\theta) - c + O(\theta)
  \]
  for some constant $c$ depending on $\eps$. We finally get that $\ell \gg e^{\log(m/2\pi)} \gg m$.

  Assume now that $m = 2$ and that there is another singular point $x'$ within $\ell(2, \eps)$ of $x$. In this case the volume of $\Omega_x$ is half that of a collar around a closed geodesic of length $r(\ell_x, \eps) \ll \eps$; as the latter is bounded below (see \cite{Halpern}) so is that of $\Omega_x$.

  \medskip

  The proof that $k_n = o(\vol O_n)$ is similar: by the Margulis lemma the regions of the $\eps$-thin part where a given conjugacy class of parabolic isometries realises the injectivity radius are pairwise disjoint, and an easy hyperbolic area computation shows that the volume of such a region is bounded below. 
\end{proof}


\section{Betti numbers of arithmetic 3--orbifolds} \label{sec_betti}

Recall that $\eps$ is a Margulis constant for $\HH^3$. Let $O$ be a 3--orbifold, then we will write $O'$ for the manifold with boundary obtained by Lemma \ref{smoothing}. We write $\Delta_\abs^1$ for the maximal self-adjoint extension of the Hodge--Laplace operator on $O'$ with absolute boundary condition. The goal of this section is to prove the following proposition, which we prove by extending the analysis at the end of section 7 in \cite{7samurai} to the orbifold case. 

\begin{prop} \label{limit_hk}
  Let $O_n$ be a sequence of closed hyperbolic 3--orbifolds which BS-converge to $\HH^3$, and let $O_n'$ be the smoothings described in Lemma \ref{smoothing}. Then for all $t > 0$ we have that 
  \[
  \limsup_{t \to +\infty} \lim_{n \to +\infty} \frac{ \otr(e^{-t\Delta_\abs^1[O_n']}) }{ \vol O_n } = 0.
  \]
\end{prop}

Before giving the proof we explain how this implies Theorem \ref{Betti}: let $O_n = \Gamma_n \bs \HH^3$. By Hodge theory we have $b_1(O_n') \le \otr(e^{-t\Delta_\abs^1[O_n']})$ for all $t$, and so Proposition \ref{limit_hk} implies that
\[
\lim_{n \to +\infty} \frac{b_1(O_n')}{\vol O_n} = 0. 
\]
On the other hand we have that the orbifold fundamental group $\Gamma_n$ is a quotient of $\pi_1(O_n')$. Indeed, the universal cover of $(O_n)_{\ge \eps}$ is a cover of the connected subset $(\widehat O_n)_\eps$ of those $x \in \HH^3$ which are not displaced by less than $\eps$ by some non-trivial element of $\Gamma_n$, and $(O_n)_{\ge \eps}$ is homotopy equivalent to $O_n'$. Moreover $O_n'$ is aspherical (as the cover $(\widehat O_n)_\eps$ constructed above is) so that $H_1(O_n')$ is the abelianisation of $\pi_1(O_n')$. From these two facts it follows that $b_1(\Gamma_n) \le b_1(O_n')$, so that $b_1(\Gamma_n) = o(\vol O_n)$ as well. 

\medskip

The proof of Proposition \ref{limit_hk} is done in three steps: first we observe convergence of the part of the trace formula for $O_n$ coming from the $\eps$-thick part: see \eqref{lim_hk_thick}. The two next steps together imply that the trace of the heat kernel on $O_n'$ is asymptotically the same as that computed in \eqref{lim_hk_thick}: first we analyse the integral of the difference on the $R$-thick part and show that it limit superior is $o(R)$ (see \eqref{comphk_thick}, then we prove that the integral on the $R$-thin part of $O_n'$ asymptotically vanishes (see \eqref{near_bd}). Altogether these three steps imply that
\[
\lim_{n \to +\infty} \frac{ \otr(e^{-t\Delta_\abs^1[O_n']}) }{ \vol O_n } = \tr e^{-t\Delta^1[\HH^3]}
\]
where we denoted $\tr e^{-t\Delta^1[\HH^3]} = \tr e^{-t\Delta^1[\HH^3]}(\tilde x, \tilde x)$ for any $\tilde x \in \HH^3$. The proposition now follows from the vanishing of the first $L^2$-Betti number of $\HH^3$, which means that $\lim_{t \to +\infty} \tr e^{-t\Delta^1[\HH^3]} = 0$ (see \cite{Lueck_book}).


\subsection{Trace formula on the thick part}

Let $O_n$ be a sequence as in Proposition \ref{limit_hk}. We prove here that
\begin{equation} \label{lim_hk_thick}
  \int_{(O_n)_{\ge\eps}} \tr e^{-t\Delta^1[O_n]}(x, x) dx - \tr e^{-t\Delta^1[\HH^3]} \cdot \vol O_n = o(\vol O_n).
\end{equation}
Let $\mathcal C_{n, e}$ and $\mathcal C_{n, h}$ be the sets of conjugacy classes of respectively elliptic and hyperbolic elements in $\Gamma_n$. For $\gamma \in \Gamma$ let $\mathcal F_\gamma$ be a fundamental domain for the centraliser $Z_\gamma$ of $\gamma$ in $\Gamma$ and $\mathcal F_\Gamma^{\ge \eps}$ the part of it on which no non-trivial element of $\Gamma$ displaces by less than $\eps$. The proof of the Selberg trace formula then gives that
\begin{multline} \label{trace_formula}
  \int_{(O_n)_{\ge\eps}} \tr e^{-t\Delta^1[O_n]}(x, x) dx  = \vol(O_n)_{\ge\eps}\tr e^{-t\Delta^1[\HH^3]} \\ + \sum_{[\gamma] \in \mathcal C_{n, e} \cup \mathcal C_{n, h}} \int_{\mathcal F_\gamma^{\ge\eps}} \tr(\gamma^* e^{-t\Delta^1[\HH^3]} (x, \gamma x)) dx.
\end{multline}
Because of Benjamini--Schramm convergence we have $\vol O_n - \vol(O_n)_{\ge\eps} = o(\vol O_n)$. Then \eqref{lim_hk_thick} will follow from \eqref{trace_formula} together with the following limit: 
\begin{equation} \label{vanishing_contrib}
  \sum_{[\gamma] \in \mathcal C_{n, e} \cup \mathcal C_{n, h}} \int_{\mathcal F_\gamma^{\ge\eps}} \tr(\gamma^* e^{-t\Delta^1[\HH^3]} (x, \gamma x)) dx = o(\vol O_n).
\end{equation}
We proceed to prove (\ref{vanishing_contrib}). The proof for the hyperbolic part is exactly the same as in \cite[Section 7]{7samurai}.

We deal now with the elliptic part; to simplify notation we cheat slightly by integrating over the part of $\mathcal F_\gamma$ where elliptic elements in $Z_\gamma$ translate by at least $\eps$, which we will continue to denote by $\mathcal F_\gamma^{\ge\eps}$ (note that it is larger than what we denoted by $\mathcal F_\gamma^{\ge \eps}$ above). If $[\gamma]$ is an elliptic conjugacy class we let $\theta_\gamma$ be its rotation angle and $\ell_\gamma$ the minimal translation length of an hyperbolic isometry in $Z_\gamma$. Then we have by integrating in polar coordinates around the axis of $\gamma$ that
\[
\int_{\mathcal F_\gamma^{\ge\eps}} \tr(\gamma^* e^{-t\Delta^1[\HH^3]} (x, \gamma x)) dx = \theta_\gamma\ell_\gamma \int_{\max(\ell(\theta_\gamma, \eps), r(\ell_\gamma, \eps))}^{+\infty}  f_\theta(r) dr
\]
where $f_\theta(r) = \sinh(r)\cosh(r)\tr(\gamma^*e^{-t\Delta[\HH^3]}(x, \gamma x))$ for a point $x$ at distance $r$ from the axis. This is a consequence of desintegration of hyperbolic volume in cylindrical coordinates \cite[p. 205]{fenchel1}. Let $\Sigma_n$ be the set of singular geodesics in $O_n$ (so each is the image of an axis of an elliptic conjugacy class in $\Gamma_n$). If $\gamma$ is an elliptic isometry of order $m$, primitive in $\Gamma$, there are $m-1$ elliptic elements in $Z_\gamma$ sharing the same axis. So we get that
\[
\sum_{[\gamma] \in \mathcal C_{n, e}} \int_{\mathcal F_\gamma^{\ge\eps}} \tr(\gamma^* e^{-t\Delta^1[\HH^3]} (x, \gamma x)) dx = \sum_{c \in \Sigma} 2\pi \ell_c \frac{o_c - 1}{o_c} \int_{\max(\ell(\theta_\gamma, \eps), r(\ell_\gamma, \eps))}^{+\infty} f_{2\pi/o_c}(r) dr
\]
where $\ell_c$ is the length of $c$ and $2\pi/o_c$ its cone angle. By the Gaussian estimate of the heat kernel of $\HH^3$ we have that
\[
f_{2\pi/o_c}(r) \ll C(t)e^{-c(t)r^2}
\]
uniformly for $r \ge \ell(\theta_\gamma, \eps)$ and it follows that
\[
\sum_{[\gamma] \in \mathcal C_{n, e}} \int_{\mathcal F_\gamma^{\ge\eps}} \tr(\gamma^* e^{-t\Delta^1[\HH^3]} (x, \gamma x)) dx \ll \sum_{c \in \Sigma_n} \ell_c
\]
and the right-hand side is an $o(\vol O_n)$ by Benjamini--Schramm convergence.


\subsection{Comparison between heat kernels}

We prove here that
\begin{equation} \label{comphk_thick}
  \lim_{R \to +\infty} \limsup_{n \to +\infty} \frac 1 {\vol O_n} \int_{(O_n)_{\ge R}} \tr (e^{-t\Delta^1[O_n]} - e^{-t\Delta^1[O_n']})(x, x) dx = 0. 
\end{equation}
To do this we let $U_n$ be the subset of $\HH^3$ covering $O_n'$ and choose a fundamental domain $D_n$ for $\Gamma$ acting in the subset of $U_n$ covering $(O_n)_{\ge R}$ (we assume $R$ is large enough so that $(O_n)_{\ge R} \subset O_n'$). Then we can write
\begin{align*}
  \int_{(O_n)_{\ge R}} \tr (e^{-t\Delta^1[O_n]} - e^{-t\Delta^1[O_n']})(x, x) dx &= \int_{D_n} \sum_{\gamma\in \Gamma} \tr\gamma^*(e^{-t\Delta^1[\HH^3]} - e^{-t\Delta[U_n]})(x, \gamma x) dx \\
  &\ll e^{-\frac{R^2}{Ct}} \int_{D_n} \sum_{\gamma\in \Gamma} e^{-\frac{d(x, \gamma x)^2}{Ct}} dx
\end{align*}
where the second line follows from \cite[Theorem 2.26]{Lueck_Schick}. By the same arguments as used above to demonstrate \eqref{lim_hk_thick} the integral is $O(\vol O_n)$ (with a constant independent of $R$ as the domain of integration shrinks when we take $R$ to infinity). In the end we get that
\[
\limsup_{n \to +\infty} \frac 1 {\vol O_n} \int_{(O_n)_{\ge R}} \tr (e^{-t\Delta^1[O_n]} - e^{-t\Delta^1[O_n']})(x, x) dx \ll e^{-\frac{R^2}{Ct}}
\]
from which \eqref{comphk_thick} follows immediately. 


\subsection{Heat kernel near the boundary}

Here we prove the final ingredient for the proof of Proposition \ref{limit_hk}: for all $R > 0$ we have 
\begin{equation} \label{near_bd}
  \int_{O_n' \setminus (O_n)_{\ge R}} \tr e^{-t\Delta^1[O_n']}(x, x) dx = o(\vol O_n). 
\end{equation}
By Benjamini--Schramm convergence we have that $\vol(O_n' \setminus (O_n)_{\ge R}) = o(\vol O_n)$. So to prove \eqref{comphk_thick} it suffices to see that $\tr e^{-t\Delta^1[O_n']}(x, x) = O_t(1)$ for $x \in O_n'$. As in \cite[(7.19.4)]{7samurai} this follows from \cite[Theorem 2.35]{Lueck_Schick}; the latter is applicable with a uniform constant in our context by Lemma \ref{smoothing}.


\appendix

\section{Proof of Lemma \ref{dist_trans}} \label{proof_lemma}

  Let $x, y \in X$. As $\gamma$ is hyperbolic there exists $a, c$ such that $L = \langle \gamma\rangle x$ is a $(c, a)$-quasi-geodesic. Regarding the conclusion of the proposition it does not change anything if we assume that $x$ is the approximate projection of $y$ on $L$, meaning that any point $x'$ of $L$ within distance $d(y, L)$ of $y$, satisfies $d(x', x) \le K$ (where $K$ depends only on the hyperbolicity constant $\delta$). 

  Let $\ell = d(x, \gamma x)$. Note first that if $k$ is large enough so that
  \begin{equation} \label{k_large}
    k > 100c\ell^{-1}K\log(k) + ac
  \end{equation}
  holds, and $y$ is close enough to $L$ so that
  \begin{equation} \label{y_away}
    d(y, x) > c^2\ell^{-1}\log(k) + cK(2+\log(2+k)) + ca
  \end{equation}
  does not then the conclusion is immediate by the triangle inequality. Thus from now on we will assume that both hold. 
  
  Let $x_i = \gamma^i x$, $y_i = \gamma^i y$ for $0 \le i \le k$. Let $F$ be the finite set
  \[
  F = \{x_0, x_1, \ldots, x_k \} \cup \{ y_0, y_k \};
  \]
  by \cite[Proposition 7.3.1]{Bowditch} there exists a choice of a ``spanning tree'' on $F$ (that is, a tree whose edges are a subset of all pairs of geodesics segment between points of $F$) such that
  \begin{equation} \label{spanning_tree}
    \forall p, q \in F :\: d(p, q) \ge d_{T_F}(p, q) - (1 + \log(2+k))K
  \end{equation}
  where $K$ depends only on $\delta$ (so we take it equal to the $K$ introduced above to simplify notation). One of $y_0, y_k$ must be connected to one of the $x_i$ in $T_F$; we may assume that $[y_0, x_i]$ is an edge in $T_F$ for some $i$. We claim that this $i$ must be unique, and we must have
  \begin{equation} \label{i_small}
    i < c\ell^{-1}\left( (\log(k+2) + 2)K +a \right). 
  \end{equation}
  Indeed, let $i$ be the smallest integer such that $[x_i, y_0] \subset T_F$. Then, because
  \[
  d_{T_F}(x_0, y_0) \le d(x_0, y_0) + (\log(k+2) + 1)K
  \]
  and
  \[
  d_{T_F}(x_0, y_0) \ge d(x_0, x_i) + d(x_i, y_0) \ge \frac{i \ell} c - a + d(x_0, y_0) - K
  \]
  we see that $i$ must verify \eqref{i_small}. Now assume that there is a $j > i$ such that $[x_j, y_0] \subset T_F$, and take it to be the smallest such; we want to reach a contradiction. 
  Consider $i \le l < j$ to be maximal such that the path in $T_F$ from $x_l$ to $x_i$ does not go through $y_0$. Then the path in $T_F$ from $x_l$ to $x_{l+1}$ must go through $y_0$ (otherwise we would have a path from $x_{l+1}$ to $x_i$ via $x_l$ avoiding $y_0$). We have thus $d_{T_F}(x_l, x_{l+1}) \ge d(x_0, y_0) - K$ which together with \eqref{y_away} and \eqref{spanning_tree} contradicts the fact that $d(x_l, x_{l+1}) = \ell$. 

  We now want to prove that $[y_0, y_k]$ is not an edge in $T_F$. To do so we must consider two possibilities. Assume first that $[y_k, x_j] \subset T_F$ for some $j$. Then reasoning as above we see that $j$ is the only such index, and $j > k - c\ell^{-1}\left( (\log(k+2) + 2)K +a \right) > i$. In this case we reach a contradiction in the same way as in the previous paragraph: considering a maximal $i \le l < j$ such that the path from $x_l$ to $x_i$ does not go through $y_0$ we see that $d_{T_F}(x_l, x_{l+1})$ is too large.

  If there is no edge $[y_k, x_j]$ in $T_F$ then the path from $x_k$ to $y_k$ must go first to $x_i$, then to $y_0$ and finally to $y_k$. But as $d(x_k, x_i) > (\log(k+2) + 1)K$ by \eqref{i_small} and \eqref{k_large} we see that this contradicts $d(x_0, y_0) = d(x_k, y_k)$. 

  So we get that there must be a unique edge $[y_k, x_j]$ in $T_F$, and the path in $T_F$ from $y_0$ to $y_k$ must go through $x_j$ and $x_i$. As before we must have
  \[
  j > k - c\ell^{-1}\left( (\log(k+2) + 2)K +a \right)
  \]
  and we finally get using first \eqref{spanning_tree}, then the fact that $(x_0, \ldots, x_k)$ is a quasi-geodesic, and finally the above together with \eqref{i_small} that:
  \begin{align*}
    d(y_0, y_k) &\ge d(y_0, x_i) + d(x_i, x_j) + d(x_j, y_k) - K - K\log(2 + k)\\
    &\ge 2d(x_0, y_0) + c^{-1}(j-i)\ell - a - 3K - K\log(2 + k) \\
    &\ge 2d(x_0, y_0) + c^{-1}\ell k - B - b\log(k)
  \end{align*}
  where $B, b$ depend only on $x, \gamma, \delta$. From the last inequality and \eqref{k_large} the conclusion is immediate.


\section{Smoothing corners} \label{appendix_smooth}

In this appendix we prove Proposition \ref{general_smoothing}; as the argument is technical but has no subtleties we will be quite sketchy in presenting it. 

Recall that we have the following situation: $X$ is a manifold with bounded geometry, $H_1, H_2 \subset X$ such that $X \setminus H_i$ both have bounded geometry, meet transversally and the dihedral angle between them is bounded away from 0 and $\pi$. We remark that constructing a smoothing of $Y = X \setminus (H_1 \cup H_2)$ satisfying the conclusions of Proposition \ref{general_smoothing} is immediate in the case where the intersection $I = H_1 \cap H_2$ has a neighbourhood in $Y$ which is isometric to the product $[0, \delta[^2 \times I$. In general we will prove the following statement: there exists a diffeomorphism $\varphi$ from $[0, \delta[^2 \times I$ to a neighbourhood of $I$ in $Y$ such that $\varphi$ and $\varphi^{-1}$ have all their derivatives uniformly bounded. In view of the preceding remark this proves the proposition.

To define $\varphi$ we need some more auxiliary notation: for a vector field $V$ and $t \ge 0$ we let $\Phi_V^t$ be its flow at time $t$; if $H \subset Z$ is open with smooth boundary we denote by $N_H^Z$ the normal field of $H$ in $Z$. We put: 
\[
\varphi_1(x, t, s) = \Phi_{N_{H_1}^X}^t(\Phi_{N_I^{H_1}}^s(x))
\]
and
\[
\varphi_2(x, t, s) = \Phi_{N_{H_2}^X}^s(\Phi_{N_I^{H_2}}^t(x))
\]
We fix a smooth non-decreasing function $h : \RR \to [0, 1[$ such that $h$ is zero on negative numbers, and at infinity it tends to 1 and all its derivatives vanish at all orders. Let $0 < a <1$ such that the convex hull of all $\varphi_1(x, t, s)$ and $\varphi_2(x, t, s)$ for $as \le t \le a^{-1}s$ is contained in $Y$. For $x, y \in X$ and $u \in [0, 1]$ let $ux + (1-u)y$ denote the barycenter of $x, y$ on the geodesic segment between them\footnote{This is well-defined for those pairs of points in $X$ that we consider, as long as we take $\delta \ll \mathrm{inj}(X)$}. With this notation we define:
\[
\varphi(x, t, s) = h\left( \frac{at - s}{as - t}\right) \varphi_1(x, t, s) + \left(1 - h\left( \frac{at - s}{as - t}\right) \right) \varphi_2(x, t, s)
\]
and we claim that $\varphi$ has the desired properties. It is smooth as a composition of smooth maps. To deduce the remaining properties we will use the following lemma.

\begin{lem} \label{angle}
  For $i=1, 2$ there is $c$ depending only on the bounds on the geometry of $H_i$ such that the following properties hold.
  \begin{enumerate}
  \item \label{angle2} Let $z \in \pl H_i$ and $0 \le t \le \delta$. The linear map $D_z\Phi_{N_{H_i}^X}^t$ is $c$-Lipschitz on angles. The same holds for $x \in I$ and $D_x\Phi_{N_I^{H_i}}^t$. 

  \item \label{angle1} For all $x \in I$ and all $0 \le s, t < \delta$, let $y = \Phi_{N_{H_i}^X}^t(\Phi_{N_I^{H_i}}^s(x))$. Let $\gamma$ be the geodesic (in $X$) from $x$ to $y$, $u_i$ the parallel transport along $\gamma$ of the outward normal vector to $H_i$ at $x$ and $v_i = \left. \frac{\pl}{\pl\tau} \right|_{\tau=t} \Phi_{N_{H_i}^X}^{\tau}(\Phi_{N_I^{H_i}}^s(x))$. Then the angle between $u_i$ and $v_i$ is at most $c\delta$.
  \end{enumerate}
\end{lem}

\begin{proof}
  \eqref{angle2} follows from the boundedness of coefficients of the metric tensor and its inverse in normal exponential coordinates (in both $I \subset H_i$ and $\pl H_i \subset X$). \eqref{angle1} follows from \eqref{angle2}, together with the fact that parallel transport along a closed curve stays close to the identity within the $\delta$-neighbourhood. 
\end{proof}

Let $V_i$ be the vector fields given by the vectors $v_i$ defined in the lemma. As for any $x \in I$ we have that the angle between $V_1(x)$ and $V_2(x)$ lies in $[\alpha_0, \pi - \alpha_0]$ it follows from \eqref{angle1} that if we choose $\delta < c^{-1}\alpha_0/2$ we have that the angle between $V_1$ and $V_2$ at any point $x$ in the $\delta$-neighbourhood of $I$ lies in $[\alpha_0/2, \pi - \alpha_0/2]$. In particular $V_1, V_2$ define a plane field, and we define $J$ to be its orthogonal.

Let $\pi_J$ be orthogonal projection on $J$. The block decomposition of $D\varphi$ according to $TX = J \oplus (V_1+V_2)$ is:
\[
D_{(x, t, s)}\varphi = \begin{pmatrix} \pi_JD_x\varphi & C \\(1-\pi_J)D_x\varphi & B \end{pmatrix}. 
\]
We need to prove that: 
\begin{enumerate}
\item $D_x\varphi, B$ and $C$ have bounded coefficients (in terms of the bounds on the geometry);

\item $\pi_J D_x\varphi$ and $B$ are everywhere invertible and their inverses are bounded.
\end{enumerate}
Indeed, this shows that the map $\varphi$ has a derivative which everywhere invertible. In particular, it is a local diffeomorphism and as it is the identity on $I$ it is also a global diffeomorphism. This also implies that its derivative is uniformly bounded in terms of the geometry of $H_i$ and $\alpha_0$, and so is its inverse. 

\medskip

We deal first with $D_x\varphi$. We note that
\[
(D_x \varphi)_{(x, t, s)} = h\left( \frac{at - s}{as - t}\right) D_x\varphi_1(x, t, s) + \left(1 - h\left( \frac{at - s}{as - t}\right) \right) D_x\varphi_2(x, t, s) + O(\delta)
\]
because of bounded geometry and the fact that to obtain $\varphi$ we move $\varphi_1$ and $\varphi_2$ by at most $\delta$. It follows that $D_x\varphi$ is bounded. By point \eqref{angle2} of the Lemma we have that at all points the angle between the image of $D_x\varphi$ and $V_i$ is at most $c\delta$; it follows that $\|(1-\pi_J)D_x\varphi\| \ll \delta$. Moreover $D_x\varphi$ is everywhere invertible with bounded inverse, because both $A_1 = D_x\varphi_1$ and $A_2 = D_x\varphi_2$ are, and for $w \in T_xI$ the vectors $A_1(w), A_2(w)$ have an angle $\le c\delta$ between them by \eqref{angle2}.  

We also have
\[
D_t\varphi = h\left( \frac{at - s}{as - t}\right) D_t\varphi_1(x, t, s) + \left(1 - h\left( \frac{at - s}{as - t}\right) \right) D_t\varphi_2(x, t, s) + O(\delta)
\]
and similarly for $D_s\varphi$, so the coefficients of $B, C$ are bounded. 

It remains to prove that $B$ is invertible and $\det(B)$ is bounded away from zero. At a point $x \in I$ we have $D_t\varphi$ and $D_s\varphi$ belong to two disjoint open convex cones in $T_xX/J_x$; by \eqref{angle1} and \eqref{angle2} this remains true in the $\delta$-neighbourhood and the angle between the cones remains bounded away from zero, hence the matrix $B$ is invertible with uniformly bounded inverse.


\bibliographystyle{plain}
\bibliography{bib}

\begin{thebibliography}{10}

\bibitem{7samurai}
Miklos Abert, Nicolas Bergeron, Ian Biringer, Tsachik Gelander, Nikolay
  Nikolov, Jean Raimbault, and Iddo Samet.
\newblock On the growth of {$L^2$}-invariants for sequences of lattices in
  {L}ie groups.
\newblock {\em Ann. of Math. (2)}, 185(3):711--790, 2017.

\bibitem{Beardon}
Alan~F. Beardon.
\newblock {\em The geometry of discrete groups}, volume~91 of {\em Graduate
  Texts in Mathematics}.
\newblock Springer-Verlag, New York, 1983.

\bibitem{BMP}
Michel Boileau, Sylvain Maillot, and Joan Porti.
\newblock {\em Three-dimensional orbifolds and their geometric structures},
  volume~15 of {\em Panoramas et Synth\`eses [Panoramas and Syntheses]}.
\newblock Soci\'{e}t\'{e} Math\'{e}matique de France, Paris, 2003.

\bibitem{BHC}
Armand Borel.
\newblock Arithmetic subgroups of algebraic groups.
\newblock {\em Annals of mathematics}, pages 485--535, 1962.

\bibitem{Bowditch}
B.~H. Bowditch.
\newblock Notes on {G}romov's hyperbolicity criterion for path-metric spaces.
\newblock In {\em Group theory from a geometrical viewpoint ({T}rieste, 1990)},
  pages 64--167. World Sci. Publ., River Edge, NJ, 1991.

\bibitem{Bridson_Haefliger}
Martin~R. Bridson and Andr{\'e} Haefliger.
\newblock {\em Metric spaces of non-positive curvature}, volume 319 of {\em
  Grundlehren der Mathematischen Wissenschaften}.
\newblock Springer-Verlag, 1999.

\bibitem{Clozel_tau}
Laurent Clozel.
\newblock D\'{e}monstration de la conjecture {$\tau$}.
\newblock {\em Invent. Math.}, 151(2):297--328, 2003.

\bibitem{dlHarpe1}
Pierre de~la Harpe.
\newblock Spaces of closed subgroups of locally compact groups.
\newblock {\em ArXiv e-prints}, jul 2008.

\bibitem{fenchel1}
Werner Fenchel.
\newblock {\em Elementary Geometry in Hyperbolic Space}, volume~11.
\newblock Walter de Gruyter, 1989.

\bibitem{fraczyk}
M.~{Fraczyk}.
\newblock {Strong Limit Multiplicity for arithmetic hyperbolic surfaces and
  $3$-manifolds}.
\newblock {\em ArXiv e-prints}, December 2016.

\bibitem{Gelander_ICM}
Tsachik Gelander.
\newblock A view on invariant random subgroups and lattices, 2018.
\newblock To appear in proceedings of ICM 2018.

\bibitem{Gromov_hyp}
M.~Gromov.
\newblock Hyperbolic groups.
\newblock In {\em Essays in group theory}, volume~8 of {\em Math. Sci. Res.
  Inst. Publ.}, pages 75--263. Springer, New York, 1987.

\bibitem{Halpern}
Noemi Halpern.
\newblock A proof of the collar lemma.
\newblock {\em Bull. London Math. Soc.}, 13(2):141--144, 1981.

\bibitem{Kazhdan1977}
David Kazhdan.
\newblock Some applications of the {W}eil representation.
\newblock {\em J. Analyse Mat.}, 32:235--248, 1977.

\bibitem{klenke}
Achim Klenke.
\newblock {\em Probability theory}.
\newblock Universitext. Springer, London, second edition, 2014.
\newblock A comprehensive course.

\bibitem{levit}
Arie Levit.
\newblock On the benjamini-schramm limit of congruence subgroups in products.
\newblock {\em preprint}, 2017.

\bibitem{Li_Millson}
Jian-Shu Li and John~J. Millson.
\newblock On the first {B}etti number of a hyperbolic manifold with an
  arithmetic fundamental group.
\newblock {\em Duke Math. J.}, 71(2):365--401, 1993.

\bibitem{LMR_genus_zero}
D.~D. Long, C.~Maclachlan, and A.~W. Reid.
\newblock Arithmetic {F}uchsian groups of genus zero.
\newblock {\em Pure Appl. Math. Q.}, 2(2, Special Issue: In honor of John H.
  Coates. Part 2):569--599, 2006.

\bibitem{Lueck_Schick}
W.~L{\"u}ck and T.~Schick.
\newblock {$L^2$}-torsion of hyperbolic manifolds of finite volume.
\newblock {\em Geom. Funct. Anal.}, 9(3):518--567, 1999.

\bibitem{Lueck_book}
Wolfgang L{\"u}ck.
\newblock {\em {$L^2$}-invariants: theory and applications to geometry and
  {$K$}-theory}, volume~44 of {\em Ergebnisse der Mathematik und ihrer
  Grenzgebiete. 3. Folge. A Series of Modern Surveys in Mathematics [Results in
  Mathematics and Related Areas. 3rd Series. A Series of Modern Surveys in
  Mathematics]}.
\newblock Springer-Verlag, Berlin, 2002.

\bibitem{Matz}
Jasmin Matz.
\newblock Limit multiplicities for {$\mathrm{PSL}_2(\mathcal O_F)$} in
  {$\mathrm{PSL}_2(\mathbb R^{r_1} \oplus \mathbb C^{r_2}$}.
\newblock {\em Arxiv e-prints}, 2017.

\bibitem{Millson}
John~J. Millson.
\newblock On the first {B}etti number of a constant negatively curved manifold.
\newblock {\em Ann. of Math. (2)}, 104(2):235--247, 1976.

\bibitem{Osin}
D.~Osin.
\newblock Invariant random subgroups of groups acting on hyperbolic spaces.
\newblock {\em Proc. Amer. Math. Soc.}, 145(8):3279--3288, 2017.

\bibitem{raimbault}
Jean Raimbault.
\newblock On the convergence of arithmetic orbifolds.
\newblock {\em Ann. Inst. Fourier (Grenoble)}, 67(6):2547--2596, 2017.

\bibitem{Shimura}
Goro Shimura.
\newblock {\em Introduction to the arithmetic theory of automorphic functions}.
\newblock Publications of the Mathematical Society of Japan, No. 11. Iwanami
  Shoten, Publishers, Tokyo; Princeton University Press, Princeton, N.J., 1971.
\newblock Kan\^{o} Memorial Lectures, No. 1.

\bibitem{Thompson}
J.~G. Thompson.
\newblock A finiteness theorem for subgroups of {${\rm PSL}(2,\,{\bf R})$}
  which are commensurable with {${\rm PSL}(2,\,{\bf Z})$}.
\newblock In {\em The {S}anta {C}ruz {C}onference on {F}inite {G}roups ({U}niv.
  {C}alifornia, {S}anta {C}ruz, {C}alif., 1979)}, volume~37 of {\em Proc.
  Sympos. Pure Math.}, pages 533--555. Amer. Math. Soc., Providence, R.I.,
  1980.

\bibitem{Zograf}
P.~Zograf.
\newblock A spectral proof of {R}ademacher's conjecture for congruence
  subgroups of the modular group.
\newblock {\em J. Reine Angew. Math.}, 414:113--116, 1991.

\end{thebibliography}

\end{document}